\documentclass[a4paper,11pt,reqno]{amsart}

\usepackage[margin=1in,includehead,includefoot]{geometry}
\usepackage[utf8]{inputenc}
\usepackage[T1]{fontenc}
\usepackage{lmodern,microtype}
\usepackage{upref}
\usepackage{mathtools,amssymb,amsthm,amsmath}
\usepackage[foot]{amsaddr}
\usepackage{enumitem}
\usepackage[textsize=footnotesize,disable]{todonotes}
\usepackage{graphicx}
\usepackage{hyperref}
\usepackage{wrapfig}
\usepackage{xcolor}
\usepackage[margin=1cm]{caption}
\usepackage{bookmark}
\usepackage{framed}
\usepackage{tabularx}
\usepackage{mycommands}

\graphicspath{{./graphics/}}

\makeatletter
\let\old@setaddresses\@setaddresses
\def\@setaddresses{\bigskip{\parindent 0pt\let\scshape\relax\let\ttfamily\relax\old@setaddresses}}
\makeatother

\newtheorem{theorem}{Theorem}
\newtheorem{corollary}[theorem]{Corollary}
\newtheorem{lemma}[theorem]{Lemma}
\theoremstyle{remark}
\newtheorem{remark}[theorem]{Remark}


\hypersetup{
  pdftitle={Graphs that admit a Hamilton path are cup-stackable},
  pdfauthor={Petr Gregor, Arturo Merino, Torsten M\"utze, Francesco Verciani}
}

\newcommand{\torsten}[1]{\todo[inline,color=red!40]{\textbf{Torsten:} #1}}

\title{Graphs that admit a Hamilton path are cup-stackable}

\author{Petr Gregor}
\address[Petr Gregor]{Department of Theoretical Computer Science and Mathematical Logic, Charles University, Prague, Czech Republic}
\email{gregor@ktiml.mff.cuni.cz}

\author{Arturo Merino}
\address[Arturo Merino]{Universidad de O'Higgins, Rancagua, Chile}
\email{arturo.merino@uoh.cl}

\author{Torsten M\"utze}
\address[Torsten M\"utze]{Institut f\"ur Mathematik, Universit\"at Kassel, Kassel, Germany \& Department of Theoretical Computer Science and Mathematical Logic, Charles University, Prague, Czech Republic}
\email{tmuetze@mathematik.uni-kassel.de}

\author{Francesco Verciani}
\address[Francesco Verciani]{Institut f\"ur Mathematik, Universit\"at Kassel, Kassel, Germany}
\email{francesco.verciani@uni-kassel.de}

\thanks{This work was supported by Czech Science Foundation grant GA~22-15272S. The authors participated in the workshop `Combinatorics, Algorithms and Geometry' in March 2024, which was funded by German Science Foundation grant~522790373.}

\begin{document}
\begin{abstract}
Fay, Hurlbert and Tennant recently introduced a one-player game on a finite connected graph~$G$, which they called cup stacking.
Stacks of cups are placed at the vertices of~$G$, and are transferred between vertices via stacking moves, subject to certain constraints, with the goal of stacking all cups at a single target vertex.
If this is possible for every target vertex of~$G$, then $G$ is called \emph{stackable}.
In this paper, we prove that if $G$ admits a Hamilton path, then $G$ is stackable, which confirms several of the conjectures raised by Fay, Hurlbert and Tennant.
Furthermore, we prove stackability for certain powers of bipartite graphs, and we construct graphs of arbitrarily large minimum degree and connectivity that do not allow stacking onto any of their vertices.
\end{abstract}


\maketitle

\section{Introduction}
\label{sec:intro}

Motivated by the popular sport of speed stacking, where the goal is to quickly stack and unstack cups in various formations, a recent paper of Fay, Hurlbert and Tennant~\cite{MR4749373} introduced a one-player game on a finite connected graph~$G$, which they called \defi{cup stacking}.
Initially, there is a single cup placed on every vertex of~$G$.
Each move consists of removing all the $r$ cups placed on one vertex~$x$ and stacking them onto the cups of a vertex~$y$ in distance exactly~$r$ from~$x$, provided that $y$ already has at least one cup on it.
The objective is to eventually move all cups onto a single target vertex~$t$, in which case $G$ is called \defi{$t$-stackable}.
If $G$ is $t$-stackable for every vertex~$t$, then it is called \defi{stackable}.
Note that the number of vertices with cups on them decreases by 1 with every move, so on a stackable graph the game always ends after exactly $n-1$ moves, where $n$ is the number of vertices of~$G$.
We remark that stackability is \emph{not} a monotone property under adding edges.
In particular, there are graphs that are stackable, which by adding edges turn into a non-stackable graph.

Cup-stacking is a recent addition to the zoo of one- and two-player games where `things move around' in a finite graph, such as graph pebbling~\cite{MR4339423}, graph pegging~\cite{MR2510324}, chip firing~\cite{MR1120415}, cops and robbers~\cite{MR2830217}, etc.

The paper~\cite{MR4749373} proves that paths, cycles, two-dimensional grids, the $d$-dimensional hypercube $Q_d$ for $d\leq 20$, Kneser graphs~$K(n,k)$ for $n\geq 3k-1$, and Johnson graphs with diameter~2 are stackable.
It also characterizes which complete partite graphs are stackable.
The authors also conjectured that the hypercube~$Q_d$ is stackable for all $d\geq 1$, and more generally that Cartesian products of paths of arbitrary lengths are stackable.
They also asked whether all connected Kneser graphs and generalized Johnson graphs are stackable.

\subsection{Our results}

The main contribution of this work is to show that graphs that admit a Hamilton path are stackable (Theorem~\ref{thm:ham}).
This considerably enlarges the catalogue of graphs that are known to be stackable.
In particular, this yields stackability of $d$-dimensional hypercubes for all $d\geq 1$, Cartesian products of paths of arbitrary lengths, connected Kneser graphs and generalized Johnson graphs (Corollary~\ref{cor:impl}), resolving the problems raised by Fay, Hurlbert and Tennant~\cite{MR4749373}.

In view of this result, the remainder of our paper focusses on graphs that do not admit Hamilton paths.
To this end, we present a versatile tool, which we call `chunking lemma' (Lemma~\ref{lem:chunk}), to split paths in a bipartite graph into subpaths of the correct sizes so as to stack the cups from the subpaths to a particular target vertex.
We apply this lemma to bipartite graphs with small diameter (Theorem~\ref{thm:bip}), and powers of graphs (Theorem~\ref{thm:power}), including trees (Theorems~\ref{thm:tree3} and~\ref{thm:sub}), proving in particular that a sufficiently large power of any connected bipartite graph is stackable (Theorem~\ref{thm:high-power})

Lastly, we construct graphs that are not $t$-stackable for any target vertex~$t$, which we call \defi{strongly non-stackable}.
In particular, we obtain strongly non-stackable graphs with arbitrarily large minimum degree and connectivity (Theorems~\ref{thm:nonstack-mindeg} and~\ref{thm:nonstack-conn}, respectively).
We also construct families of graphs that are stackable, but which can be made (strongly) non-stackable by adding edges, showing that stackability is not monotone (Theorems~\ref{thm:nonmono} and~\ref{thm:strong-nonmono}).
This demonstrates that our result on stackability of graphs with a Hamilton path is not a trivial consequence of the fact that paths are stackable.

\subsection{Related games and rules}

In our version of the game, a move of $r$ cups from a vertex~$x$ to another vertex~$y$ requires that there is a shortest path~$P(x,y)$ of length~$r$ from~$x$ to~$y$ in~$G$.
This is the reason for the non-monotonicity under adding edges, as adding an edge may change shortest paths.
A similar game was studied by Veselovac~\cite{veselovac:2022}, using slightly different terminology (instead of cups being stacked, he considers frogs that jump on each other), where `shortest path $P(x,y)$' is replaced by `path $P(x,y)$', i.e., cups (or frogs) can move along non-geodesic paths\footnote{The English version of the abstract of his thesis says `shortest paths', but this is a mistranslation of the original Croatian version. Definition 1.2 in the thesis clearly requires `paths' only, not `shortest paths'.}, a variant of the game that was also mentioned in~\cite{MR4749373} for further study.
Of course, if $G$ is stackable in the geodesic game, then $G$ is also stackable in the non-geodesic game.
Furthermore, if $G$ is a tree both variants are identical, as any path in a tree is also a shortest path.
In fact, most results in Veselovac's thesis are about trees.
Specifically, he shows that paths, stars, and so-called starfishes and dandelions are stackable.
He also proves that the tree~$T$ obtained from two stars with at least 3 rays each by joining their centers with an edge is not $t$-stackable for any vertex~$t$.
The non-geodesic game behaves monotonically with respect to adding edges, namely if we have a stackable graph, then any graph obtained by adding edges is also stackable.
In particular, since the path is stackable, any graph that admits a Hamilton path is also stackable.
For the rest of this paper, we do not consider this variant of the game, but instead we focus on the game introduced by Fay, Hurlbert and Tennant~\cite{MR4749373} where cups have to be moved along geodesic (i.e., shortest) paths.

Mitchell~\cite{mitchell:2023} considers yet another variant, where `shortest path~$P(x,y)$' is replaced by `walk~$P(x,y)$', i.e., vertices and edges used during one move may repeat.
As pointed out in~\cite{MR4749373}, other variants are of course possible, such as `trail~$P(x,y)$', i.e., vertices may repeat, but not edges.

Woll~\cite{woll:17} has investigated the game on the path where in the initial configuration each vertex has either~0 or~1 cups.

\subsection{Outline of this paper}

In Section~\ref{sec:ham} we prove that the existence of a Hamilton path implies stackability.
In Section~\ref{sec:bip} we present the chunking lemma and various applications of it to bipartite graphs and their powers (that do not admit Hamilton paths).
In Section~\ref{sec:nonstack} we construct graphs that are not $t$-stackable for any of target vertex~$t$.
In Section~\ref{sec:nonmono} we show that stackability is not monotone under adding edges.
We conclude with some open questions in Section~\ref{sec:open}.

\section{Graphs that admit a Hamilton path}
\label{sec:ham}

A \defi{Hamilton path} in a graph is a path that visits every vertex exactly once.

\begin{theorem}
\label{thm:ham}
Any graph that admits a Hamilton path is stackable.
\end{theorem}

This substantially enlarges the catalogue of known stackable graphs by all graphs that are known to have a Hamilton path.
The following corollary lists only those explicity conjectured or mentioned in~\cite{MR4749373}.
To state the result, we write $P_\ell$ for the path with $\ell$ vertices.
Furthermore, the \defi{Cartesian product} $G\Cprod H$ of two graphs $G=(V,E)$ and $H=(W,F)$ has the vertex set~$V\times W$ and and an edge $((v,w),(v',w'))$ whenever $v=v'$ and~$(w,w')\in F$ or $(v,v')\in E$ and~$w=w'$.
The \defi{$d$-dimensional hypercube~$Q_d$} is the graph~$\Cprod_{i=1}^d P_2$.
It is easy to show by induction that Cartesian products of paths of arbitrary lengths, in particular $Q_d$, admit a Hamilton path (see~\eqref{eq:HP1p} below).
For integers~$k\geq 1$ and~$n\geq 2k+1$, the \defi{Kneser graph~$K(n,k)$} has as vertices all $k$-element subsets of~$\{1,\ldots,n\}$, and edges between disjoint sets.
The \defi{generalized Johnson graph~$J(n,k,s)$} has as vertices all $k$-element subsets of~$\{1,\ldots,n\}$, and an edge between any two sets~$A$ and~$B$ that satisfy~$|A\cap B|=s$.
To ensure that the graph is connected, we assume that $s<k$ and $n\geq 2k-s+\mathbf{1}_{[s=0]}$, where $\mathbf{1}_{[s=0]}$ denotes the indicator function that equals~1 if $s=0$ and~0 otherwise.
Clearly, we have $J(n,k,0)=K(n,k)$.
Kneser graphs and generalized Johnson graphs were recently shown to have a Hamilton cycle in~\cite{MR4617441}, with the only exception of the Petersen graph~$K(5,2)=J(5,2,0)=J(5,3,1)$, which has a Hamilton path, but not cycle.

\begin{corollary}
\label{cor:impl}
The following families of graphs are stackable:
\begin{itemize}[itemsep=0ex,parsep=0.5ex,leftmargin=4ex]
\item $d$-dimensional hypercubes~$Q_d$ for all $d\geq 1$;
\item Cartesian product of paths $P_{\ell_1}\Cprod P_{\ell_2}\Cprod \cdots \Cprod P_{\ell_d}$ for all $d\geq 1$ and all integers~$\ell_1,\ldots,\ell_d\geq 1$;
\item Kneser graphs~$K(n,k)$ for all $k\geq 1$ and $n\geq 2k+1$;
\item generalized Johnson graphs~$J(n,k,s)$ for all $k\geq 1$, $0\leq s<k$, and $n\geq 2k-s+\mathbf{1}_{[s=0]}$.
\end{itemize}
\end{corollary}

Theorem~\ref{thm:ham} is an immediate consequence of the following lemma.

\begin{lemma}
\label{lem:path}
Let $G$ be a graph and $P=(x_1,\ldots,x_\ell)$ a subpath of~$G$.
Then for any~$t\in\{1,\ldots,\ell\}$, we can stack the $\ell$ cups from~$P$ onto the vertex~$x_t$.
\end{lemma}

We write $d(x,y)$ for the distance between vertices~$x$ and~$y$ in~$G$.

\begin{proof}
We argue by induction on~$\ell$.
The induction basis~$\ell=1$ is trivial.
For the induction step let~$\ell\geq 2$.
Our task is to stack the $\ell$ cups from $P=(x_1,\ldots,x_\ell)$ onto the target vertex~$x_t$.
We assume w.l.o.g.\ that $t\geq 2$, otherwise we can do the argument with the reversed path.
We define $s:=d(x_1,x_t)$, and we split $P$ into two paths~$P':=(x_1,\ldots,x_s)$ and~$P'':=(x_{s+1},\ldots,x_\ell)$.
Note that $s\leq t-1$ and therefore $x_t\in P''$.
By induction, we can stack the $s$ cups from~$P'$ onto~$x_1$.
Similarly, by induction, we can stack the $\ell-s$ cups from~$P''$ onto~$x_t$.
In the last step, because of~$d(x_1,x_t)=s$, we can move the $s$ cups from~$x_1$ to~$x_t$.
This completes the proof.
\end{proof}

Note that this proof does not require that~$P$ is an isometric path in~$G$, i.e., the distances between vertices on~$P$ along the path need not be the same as the distances between those vertices in the whole graph~$G$.

By Theorem~\ref{thm:ham}, a graph is stackable if it has a Hamilton path.
The converse implication, however, is not true in general.
With the help of a computer, we found the smallest graphs that are stackable but do not admit Hamilton paths.
They have 6~vertices and are all shown in Figure~\ref{fig:nonham}.

\begin{figure}[h]
\includegraphics{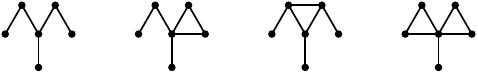}
\caption{All smallest stackable graphs that do not admit Hamilton paths.}
\label{fig:nonham}
\end{figure}

\section{Bipartite graphs}
\label{sec:bip}

In this section we establish stackability of various bipartite graphs, in particular graph powers, that do not have Hamilton paths; see Remark~\ref{rem:diff} below.

We will use the following important property of bipartite graphs, stated without proof.

\begin{lemma}
\label{lem:dpm1}
Let $G$ be a bipartite graph, let $t$ be one of its vertices, and let $xy$ be an edge of~$G$.
Then we have $d(x,t)=d(y,t)\pm 1$.
\end{lemma}

\subsection{Chunking lemma}

For integers~$a\leq b$ we define $[a,b]:=\{a,a+1,\ldots,b\}$.
We also write~$[N]:=[1,N]=\{1,\ldots,N\}$.
Given a sequence $x_1,\ldots,x_N$ of integers, a \defi{chunk} is a consecutive subsequence $x_i,x_{i+1},\ldots,x_j$ for $1\leq i\leq j\leq N$, and we refer to the quantity~$j-i+1$ as its \defi{length}.
Such a chunk is \defi{proper} if the value of one of its element equals its length, i.e., $x_k=j-i+1$ for some $k\in[i,j]$.
For example, the sequence $6,5,3,4,5,6,4$ can be partitioned into the chunks~$6,5,3$ and~$4,5,6,4$, both of which are proper, and this is in fact the only way to split this sequence into proper chunks.

\begin{lemma}
\label{lem:chunk}
Let $x_1,\ldots,x_N$ be a sequence of positive natural numbers such that $(x_1,x_2)\neq (2,1)$ and $x_{i+1}=x_i\pm 1$ for $i=1,\ldots,N-1$.
If $N\geq (\max_{i=1}^N x_i)^2$, then the sequence can be partitioned into proper chunks.
\end{lemma}

The assumption $(x_1,x_2)\neq (2,1)$ is needed, because the sequence $2,1,2,1,\ldots,2,1,2$ cannot be partitioned into proper chunks.

We will apply Lemma~\ref{lem:chunk} to partition a path~$P$ with $N$ vertices in a bipartite graph~$G$ into subpaths.
The number~$x_i$ is the distance of the $i$th vertex on~$P$ to some fixed target vertex~$t$ not on the path, and by Lemma~\ref{lem:dpm1} these numbers alternate by $\pm 1$ by the assumption that $G$ is bipartite.
The proper chunks correspond to subpaths of~$P$ with the property that the number of vertices on the subpath equals the distance of one vertex~$s$ on the subpath to~$t$.
We can then stack all cups from the subpath onto~$s$ by Lemma~\ref{lem:path}, and from there with one additional move onto~$t$.
The crucial condition stated in the lemma for this to work is that $N$ is at least quadratic in the largest of the distances~$x_i$.

\begin{proof}
We say that $i\in[N]$ is a \defi{cut} if $x_1,\ldots,x_i$ can be partitioned into proper chunks.
Our goal is to prove that~$N$ is a cut.

By the assumption that $(x_1,x_2)\neq (2,1)$ and~$x_2=x_1\pm 1$, there are two consecutive cuts $\min\{x_1,x_2\},\max\{x_1,x_2\}$.

\begin{figure}
\includegraphics{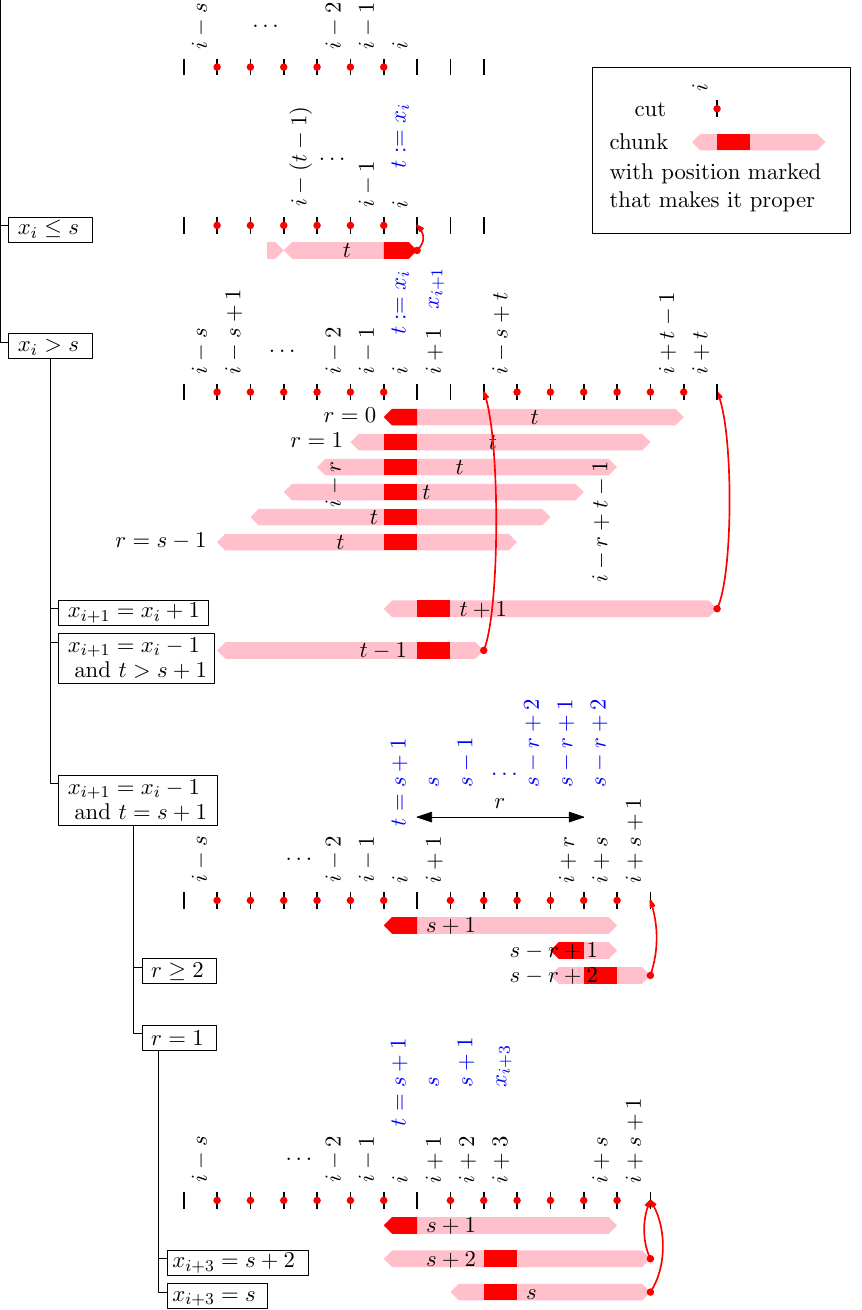}
\caption{Illustration of the proof of Lemma~\ref{lem:chunk}.}
\label{fig:chunk}
\end{figure}

Consider a sequence of $s\geq 2$ consecutive cuts $[i-s,i-1]$.
The different cases in the following proof are illustrated in Figure~\ref{fig:chunk}.
We distinguish two main cases:
If $x_i\leq s$, then we define $t:=x_i$ and we have the proper chunk $x_{i-(t-1)},\ldots,x_i$ of length~$t$, implying that $i$ is a cut, i.e., we have found $s+1$ consecutive cuts $[i-s,i]$.
It remains to consider the case $x_i>s$.
We define $t:=x_i$ and consider the proper chunks $x_{i-r},\ldots,x_{i-r+t-1}$, each of length~$t$, for $r=0,\ldots,s-1$, i.e., we have found $s$ consecutive cuts $[i-s+t,i+t-1]$.
If $x_{i+1}=x_i+1$, then we have the proper chunk $x_i,\ldots,x_{i+t}$ of length $t+1=x_{i+1}$, and thus another cut~$i+t$.
If $x_{i+1}=x_i-1$ and $t>s+1$, then we have the proper chunk $x_{i-s+1},\ldots,x_{i-s+t-1}$ (note that $i-s+t-1\geq i+1$ by the additional assumption $t>s+1$) of length~$t-1=x_{i+1}$, and thus another cut $i-s+t-1$.
Otherwise we have $x_{i+1}=x_i-1$ and $t=s+1$.
Consider the smallest $r\geq 1$ such that $x_{i+1+r}=x_{i+r}+1$.
We then have $x_i,x_{i+1},\ldots,x_{i+1+r}=s+1,s,s-1,\ldots,s-r+2,s-r+1,s-r+2$ and therefore~$r\leq s$.
We distinguish the subcases~$r=1$ and~$r\geq 2$.
If $r\geq 2$ we consider the proper chunk $x_{i+r},\ldots,x_{i+s+1}$ (note that $i+2\leq i+r\leq i+s$ by the additional assumption $r\geq 2$) of length $s-r+2=x_{i+1+r}$, which gives another cut $i+s+1=i+t$.
It remains to consider the case $r=1$, i.e., $x_{i+2}=x_i=x_{i+1}+1=s+1$.
If $x_{i+3}=x_{i+2}+1=s+2$, then consider the proper chunk $x_i,\ldots,x_{i+s+1}$ of length~$s+2=x_{i+3}$, which gives another cut $i+s+1=i+t$.
If $x_{i+3}=x_{i+2}-1=s$, then consider the proper chunk $x_{i+2},\ldots,x_{i+s+1}$ of length~$s=x_{i+3}$, which gives another cut $i+s+1=i+t$.

We define $C:=\max_{i=1}^N x_i$.
By the arguments from before, we see at least 2 consecutive cuts in the interval~$[1,C]$.
Furthermore, if there are $s\geq 2$ consecutive cuts in the interval~$[1,i]$, then there are at least $s+1$ consecutive cuts in the interval~$[1,i+C+1]$.
We conclude that there are $C$ consecutive cuts in the interval~$[1,(C-1)\cdot (C+1)]$.
Note that every number greater than this sequence of $C$ consecutive cuts is also a cut, in particular $N\geq C^2>(C+1)(C-1)$ is a cut.
This completes the proof.
\end{proof}

\subsection{Bipartite graphs with small diameter}

We refer to a spanning forest of paths in~$G$ as a \defi{path partition} of~$G$ (the paths need not be induced).
For a path~$P$ in~$G$ between vertices~$x$ and~$y$, we define $T(P):=\{x,y\}$ as the set of terminal vertices of~$P$.
The \defi{diameter} of~$G$ is the maximum of~$d(x,y)$ over all pairs of vertices $x,y$ in~$G$.

\begin{theorem}
\label{thm:bip}
Let $G$ be a connected bipartite graph with diameter~$d$ and let $t$ be one of its vertices.
If $G$ has a partition into $p\geq 2$ paths~$P_1,\ldots,P_p$ such that $t\in P_1$ and $P_j$ has at least~$d^2$ vertices and $d(T(P_j))\notin\{2,4\}$ for all $j\in[2,p]$, then $G$ is $t$-stackable.
\end{theorem}

In this theorem and its proof, and also in Theorem~\ref{thm:power} and its proof below, $P_j$, $j\in[p]$, are paths of arbitrary lengths, i.e., the index~$j$ does not refer to the number of vertices.

\begin{proof}
By Lemma~\ref{lem:path}, the cups from $P_1$ can be stacked onto~$t$.
Now consider one of the remaining paths $P_j$, $j\in[2,p]$, and let $x_1,\ldots,x_N$ be the sequence of distances from the vertices along this path to~$t$.
As $t\notin P_j$ we have $x_i\geq 1$ for all $i\in[N]$, and as $G$ is bipartite we have $x_{i+1}=x_i\pm 1$ for all $i\in[N-1]$ by Lemma~\ref{lem:dpm1}.
Furthermore, note that if $d(T(P_j))\geq 5$, then we have $\max\{x_1,x_N\}\geq 3$ and we can assume w.l.o.g.\ that $x_1\geq 3$, in particular $x_1\neq 2$.
On the other hand, if $d(T(P_j))\in\{1,3\}$, then $x_1$ and~$x_N$ have opposite parity by Lemma~\ref{lem:dpm1}, in particular $x_1\neq x_N$, so we can again assume w.l.o.g.\ that $x_1\neq 2$.
The remaining cases $d(T(P_j))\in\{2,4\}$ are excluded by the assumptions in the lemma.
Lastly, as the diameter of $G$ is~$d$ the assumption $N\geq d^2$ ensures that $N\geq (\max_{i=1}^N x_i)^2$.
Therefore, the conditions of Lemma~\ref{lem:chunk} are satisfied, and we obtain that the path~$P_j$ can be split into $q$ subpaths $Q_1,\ldots,Q_q$ such that the number of vertices of~$Q_k$ equals $d(s_k,t)$ for some vertex~$s_k\in Q_k$ for all $k\in[q]$.
We can thus apply Lemma~\ref{lem:path} for each path $Q_k$, $k\in[q]$, to stack the cups from~$Q_k$ onto~$s_k$, and from there with one additional move onto~$t$.
\end{proof}

\begin{figure}[h!]
\makebox[0cm]{ 
\includegraphics{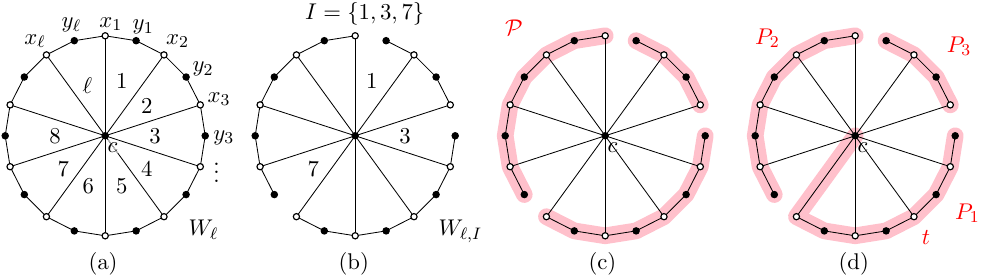}
}
\caption{Illustration of how Theorem~\ref{thm:bip} applies to certain subgraphs of a biwheel.}
\label{fig:biwheel}
\end{figure}

We proceed to give an example of how to apply Theorem~\ref{thm:bip}.
For any integer~$\ell\geq 2$ we define the \defi{biwheel} $W_\ell$ as the graph with vertex set~$\{c\}\cup \{x_1,\ldots,x_\ell\}\cup\{y_1,\ldots,y_\ell\}$ and edge set~$\{(c,x_i)\mid i\in[\ell]\}\cup \{(x_i,y_i),(y_i,x_{i+1})\mid i\in[\ell]\}$, where indices are considered modulo~$\ell$; see Figure~\ref{fig:biwheel}~(a).
For any set~$I\seq [\ell]$ we define $W_{\ell,I}:=W_\ell\setminus\{(x_i,y_i)\mid i\in I\}$, i.e., the edges~$(x_i,y_i)$ for $i\in I$ are removed from the biwheel; see Figure~\ref{fig:biwheel}~(b).
We aim to show that~$W_{\ell,I}$ is stackable provided that for any two integers~$j,j'\in I$ we have $|j-j'|\geq 8$ (modulo~$\ell$).
First note that the graph~$W_{\ell,I}$ is bipartite and has diameter~$d=4$.
Also note that~$W_{\ell,I}$ has $p:=|I|$ vertices of degree~1, so it does not have a Hamilton path for $p\geq 3$, i.e., Theorem~\ref{thm:ham} does not apply.
Clearly, the edges of $W_{\ell,I}$ not incident with the central vertex~$c$ form a set~$\cP$ of paths whose end vertices lie in different partition classes and hence satisfy~$d(T(P))\notin\{2,4\}$ for all $P\in\cP$; see Figure~\ref{fig:biwheel}~(c).
Also, each of the paths has at least $16\geq d^2$ vertices by the earlier assumption on~$I$.
For any target vertex~$t$ of~$W_{\ell,I}$, we can take~$P_1$ to be the path from~$\cP$ that contains~$t$ extended by the central vertex~$c$, or if $t=c$ we extend any of the paths by~$c$ to become~$P_1$, and we take $P_2,\ldots,P_p$ to be the remaining paths in the set~$\cP$; see Figure~\ref{fig:biwheel}~(d).
Applying Theorem~\ref{thm:bip} yields that~$W_{\ell,I}$ is $t$-stackable, as desired.

\subsection{Powers of bipartite graphs}

We now apply Lemma~\ref{lem:chunk} to powers of bipartite graphs~$G$.
The \defi{$r$th power} of a graph~$G$ is the $r$-fold Cartesian product with itself, namely $G^r:=\underbrace{G\Cprod G\Cprod\cdots\Cprod G}_{r\text{ times}}$.

\begin{remark}
\label{rem:diff}
For a connected bipartite graph~$G$, we write $\Delta(G)\geq 0$ for the difference in size between the partition classes of~$G$.
Note that any path in~$G$ omits at least~$\Delta(G)-1$ many vertices.
In particular, if~$\Delta(G)\geq 2$, then $G$ has no Hamilton path.
Furthermore, note that $G^r$ is connected and bipartite as well, and satisfies $\Delta(G^r)=\Delta(G)^r$.
Consequently, if $\Delta(G)\geq 2$, then $\Delta(G^r)\geq 2^r$, and so the longest path in~$G^r$ omits at least~$2^r-1$ vertices, and in particular does not admit a Hamilton path for any $r\geq 1$.
The results below thus yield infinitely many graphs that are stackable, but that do not admit a Hamilton path (in a very strong sense that the longest path omits at least~$2^r-1$ many vertices; so Theorem~\ref{thm:ham} does not apply).
\end{remark}

\begin{theorem}
\label{thm:power}
Let $G$ be a connected bipartite graph with diameter~$d$.
If $G$ has a partition into $p\geq 2$ paths $P_1,\ldots,P_p$ with at least~$k\geq 2$ vertices each, then for any integer $r\ge 2$ that satisfies $k^r\ge (d r)^2$ the graph~$G^r$ is stackable. 
\end{theorem}

The Cartesian product of paths~$P_1\Cprod\cdots \Cprod P_p$, where $P_p=(x_1,x_2,\ldots)$, admits a \defi{canonical Hamilton path}~$H(P_1,\ldots,P_p)$, defined recursively as $H(P_1):=P_1$ if $p=1$ and
\begin{equation}
\label{eq:HP1p}
H(P_1,\ldots,P_p):=H' x_1, \rev(H') x_2, H' x_3,\rev(H') x_4,\ldots, \quad \text{ if } p\geq 2,
\end{equation}
where $H':=H(P_1,\ldots,P_{p-1})$ and $\rev(H')$ is the path~$H'$ traversed in reverse order.
For example, for $P_1=(a,b,c)$ and $P_2=(d,e,f)$ we have
\[ H(P_1,P_2)=(a,d),(b,d),(c,d),(c,e),(b,e),(a,e),(a,f),(b,f),(c,f). \]

\begin{proof}[Proof of Theorem~\ref{thm:power}]
For $i\in[p]$ we write $N_i\geq k\geq 2$ for the number of vertices of the path~$P_i$.
The graph $G^r$ is bipartite, has diameter~$d r$, and it has an induced partition into $r$-dimensional grids $Q(i_1,\ldots,i_r):=P_{i_1}\Cprod P_{i_2}\Cprod\cdots\Cprod P_{i_r}$ for integers $i_1,\ldots,i_r\in[p]$.
Each of the grids $Q(i_1,\ldots,i_r)$ admits a canonical Hamilton path and contains $\prod_{j=1}^r N_{i_j}$ vertices.
Fix an arbitrary vertex $t=(t_1,\ldots,t_r)$ of~$G^r$.
By Lemma~\ref{lem:path}, we can stack the cups from $Q=Q(i_1,\ldots,i_r)$ with $t\in Q$ onto~$t$.
Now consider one of the remaining grids $Q=Q(i_1,\ldots,i_r)$ with $t\notin Q$.
We assume w.l.o.g.\ that $t_1\notin P_{i_1}$.
Furthermore, for each $j\in[2,r]$ we choose the orientation of the path~$P_{i_j}$ so that $t_j$ is not its first vertex (this is possible because it has $N_{i_j}\geq 2$ vertices), and we let $x_{j,1},\ldots,x_{j,N_{i_j}}$ be the sequence of distances from the vertices along this path to~$t_j$ in the graph~$G$.
By the assumption that $t_1\notin P_{i_1}$ we have $x_{1,k}\geq 1$ for all $k\in[N_{i_1}]$, and by the choice of orientation of~$P_{i_j}$ we have $x_{j,1}\geq 1$ for all $j\in[2,r]$.
Let $x_1,\ldots,x_N$, $N:=\prod_{j=1}^r N_{i_j}\geq k^r$, be the sequence of distances from the vertices along the canonical path $H=H(P_{i_1},\ldots,P_{i_r})$ to~$t$ in the graph~$G^r$.
As $t\notin Q$ we have $x_i\geq 1$ for all $i\in[N]$, and as $G^r$ is bipartite we have $x_{i+1}=x_i\pm 1$ for all $i\in[N-1]$ by Lemma~\ref{lem:dpm1}.
Furthermore, note that $x_1=\sum_{j=1}^r x_{j,1}$, which if $r\geq 3$ implies that $x_1\geq 3$ and in particular $x_1\neq 2$.
If $r=2$ and $x_1=2$, then we have $x_{1,1}=x_{2,1}=1$ and therefore $x_{1,2}=2$, implying that $x_2=x_{1,2}+x_{2,1}=3$ and therefore $(x_1,x_2)\neq (2,1)$.
Lastly, as the diameter of~$G^r$ is $d r$ the assumption $N\geq k^r \geq (d r)^2$ implies that $N\geq (\max_{i=1}^N x_i)^2$.
Therefore, the conditions of Lemma~\ref{lem:chunk} are satisfied, we can split the path~$H$ into subpaths, stack the vertices from each subpath onto one of its vertices that has the correct distance to~$t$ by Lemma~\ref{lem:path}, and from there with one additional move onto~$t$.
\end{proof}

\begin{corollary}
\label{thm:high-power}
For any connected bipartite graph~$G$ that has a partition into $p\geq 2$ paths with at least~$k\geq 2$ vertices, there is an integer~$r_0\ge 2$ such that $G^r$ is stackable for all~$r\geq r_0$.
\end{corollary}

\begin{proof}
Let $d$ be the diameter of~$G$.
Note that $k^r$ grows exponentially with~$r$, whereas $(d r)^2$ grows quadratically with~$r$, so there is an integer~$r_0$ such that $k^r\geq (d r)^2$ for all $r\geq r_0$, and then Theorem~\ref{thm:power} applies.
\end{proof}

\subsection{Powers of trees}

To find a partition of~$G$ into long paths, it suffices to find a spanning tree of~$G$ in which any two leaves have large distance.
Note that the diameter of a tree~$T$ is the maximum of~$d(x,y)$, taken over all pairs of leaves~$x$ and~$y$.
We define the \defi{spread} of a tree~$T$ as the minimum of~$d(x,y)$, taken over all pairs of distinct leaves~$x$ and~$y$. 

\begin{lemma}
\label{lem:spread}
Let $T$ be a tree with spread~$k$.
Then $T$ has a partition into paths on at least~$k/2$ vertices each.
\end{lemma}

\begin{proof}
This proof is illustrated in Figure~\ref{fig:tree}.
For any two vertices~$x$ and~$y$ in~$T$, we write $P(x,y)$ for the path from~$x$ to~$y$ in~$T$.
Consider $T$ rooted at one of its leaves~$r$.
The partition of~$T$ into paths is obtained by repeatedly removing a path from this rooted tree as follows:
If the remaining tree is already a path, then the partition consists of the entire path, which has length at least~$k$ and therefore at least~$k+1$ vertices.
Otherwise, we choose a leaf~$x$ at maximum distance from~$r$, let $y$ be the vertex on~$P(r,x)$ farthest from~$r$ that has more than one descendant, and we let $y'$ be the descendant of~$y$ on~$P(r,x)$.
The path $P(y,x)$ has length at least~$k/2$, otherwise $x$ would have distance less than~$k$ from one of the other descendant leaves of~$y$.
It follows that $P(y',x)$ has at least~$k/2$ vertices.
So we remove~$P(y',x)$ from~$T$ as part of the partition and repeat the process, with the same root~$r$.
As removing~$P(y',x)$ from the tree does not create additional leaves, this yields the desired partition.
\end{proof}

\begin{figure}[h]
\includegraphics{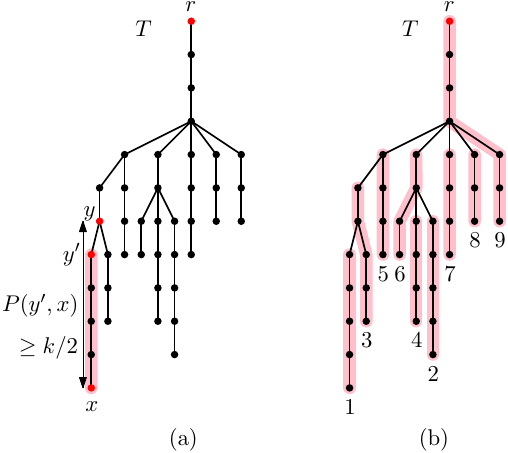}
\caption{Illustration of the proof of Lemma~\ref{lem:spread}.
Part (a) shows one step of the partition, and (b) shows the final partition of~$T$ into paths, where the numbers $1,\ldots,9$ indicate the order in which the paths are removed.}
\label{fig:tree}
\end{figure}

\begin{theorem}
\label{thm:tree3}
Let $T$ be a tree with spread~$k\geq 72$ and diameter at most~$k^{3/2}/\sqrt{72}$.
Then $T^3$ is stackable.
\end{theorem}

\begin{proof}
As $T$ is a tree, it is clearly bipartite.
Furthermore, by Lemma~\ref{lem:spread}, $T$ admits a partition into paths with at least~$k/2\geq 36$ vertices each.
For $d:=k^{3/2}/\sqrt{72}$ and $r:=3$ the inequality $(k/2)^r\geq (d r)^2$ is satisfied, as shown by the following computation:
\[ (k/2)^3 \geq \big((k^{3/2}/\sqrt{72})\cdot 3\big)^2 \quad \Longleftrightarrow \quad 1/8 \geq 9/72=1/8. \]
The claim hence follows from Theorem~\ref{thm:power}.
\end{proof}

The above proof also works, in fact, for $k\geq 3$, but the statement of the theorem would be void for $k=3,\ldots,71$, because in those cases $k>k^{3/2}/\sqrt{72}$, and the diameter of any tree is always greater or equal to the spread.

The \defi{$s$-subdivision} of a graph~$G$ is the graph obtained by replacing every edge of~$G$ by a path of length~$s$.
The next theorem shows a possible application of Theorem~\ref{thm:tree3}.
Specifically, the theorem asserts that the third power of any tree~$T$ becomes stackable when subdividing the edges of~$T$ sufficiently often.
This gives a way of constructing more infinite families of stackable powers of trees (recall Remark~\ref{rem:diff}).

\begin{theorem}
\label{thm:sub}
Let $T$ be a tree with spread~$k$ and diameter~$d$.
Let $T'$ be the $s$-subdivision of~$T$ for some $s\geq 72d^2/k^3$.
Then $(T')^3$ is stackable.
\end{theorem}

\begin{proof}
We want to apply Theorem~\ref{thm:tree3} to~$T'$.
Note that $T'$ has spread~$k'=sk$ and diameter~$d'=sd$, so the inequality $d'\leq k'^{3/2}/\sqrt{72}$ is equivalent to $sd\leq (sk)^{3/2}/\sqrt{72}$, which is equivalent to the assumption $s\geq 72d^2/k^3$.
Note also that $k'=sk\geq 72$ by the assumption $s\geq 72d^2/k^3$ and the observation that~$d\geq k$.
\end{proof}

\section{Strongly non-stackable graphs}
\label{sec:nonstack}

In this section we construct families of graphs~$G$ that are not $t$-stackable for \emph{any} target vertex~$t$ of~$G$.
Recall that we refer to such a graph as \defi{strongly non-stackable}.
Recall that Veselovac~\cite{veselovac:2022} proved that the tree obtained from two stars with at least 3 rays by joining their centers with an edge is strongly non-stackable.
In this section, we enlarge the catalogue of strongly non-stackable graphs.
In particular, we show that there are strongly non-stackable graphs with arbitrarily large minimum degree and connectivity (which must be non-trees).

The next result is a versatile result to prove non-stackability based on the presence of large independent sets in the graph.

\begin{lemma}
\label{lem:nonstack}
Let $G$ be a connected graph with vertex set~$V$ and let $t\in V$.
Let $U$ be an independent set in~$G$, and define $U':=\{x\in U\mid d(x,t)\geq 2\}$, $W:=V\setminus U$ and $d:=\max_{x\in V} d(x,t)$.
If $|U'|>(d-1)|W|$, then $G$ is not $t$-stackable.
\end{lemma}

In words, the set~$U'$ are those vertices from the independent set~$U$ that have distance at least~2 from~$t$, the set~$W$ is the complement of~$U$, and $d$ is the maximum distance of vertices from~$t$.
The inequality $|U'|>(d-1)|W|$ expresses that there are many more vertices in~$U'$ than in~$W$.
Intuitively, the graph~$G$ referred to in Lemma~\ref{lem:nonstack} has small diameter, and a large independent set, a large subset of which are non-neighbors of~$t$.

\begin{proof}
For the sake of contradiction suppose that $G$ is $t$-stackable.
Consider one of the cups placed on a vertex~$x\in U'$.
As $d(x,t)\geq 2$, this cup cannot be moved to~$t$ directly, but it is moved to~$t$ in a sequence of moves, where the final move involves $c\geq 2$ cups being moved from some vertex to~$t$.
Consider the origin of these $c$ cups.
One of them originated from~$x$ by construction.
Furthermore, not all of these $c$ cups can originate from vertices in~$U$, simply because any two vertices in~$U$ have distance at least~2, so the first move in the sequence had to involve a vertex from~$W$, which is empty after this sequence of moves. 
It follows that this sequence of moves emptied at least one vertex from~$W$ and at most $c-1$ vertices from~$U'$.
As $|U'|>(d-1)|W|$ and $d\geq c$ by the definition of~$d$, we conclude that at the end of the game, at least one of the vertices in~$U'$ still has a cup on it, a contradiction.
\end{proof}

Given a graph~$G$, the \defi{$c$-cactus of~$G$} is the graph obtained from~$G$ by attaching $c$ pendant edges to each vertex of~$G$; see Figure~\ref{fig:cactus}.

\begin{theorem}
\label{thm:cactus}
Let $G$ be a connected graph with $n\geq 2$ vertices and diameter~$d$, and consider the $c$-cactus $G'$ of~$G$ for $c>(d+1)n/(n-1)$.
Then $G'$ is strongly non-stackable.
\end{theorem}

Figure~\ref{fig:cactus} shows examples of applying Theorem~\ref{thm:cactus} to construct strongly non-stackable graphs.

\begin{figure}[h]
\includegraphics[page=1]{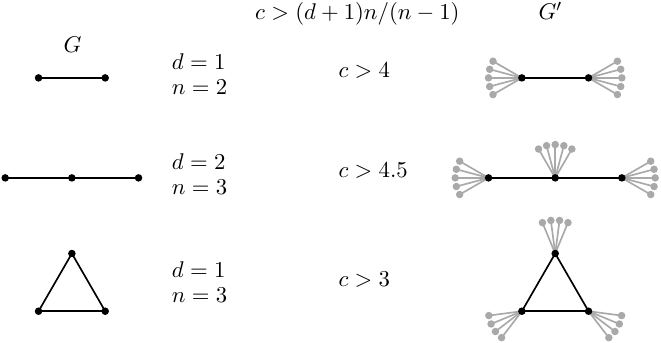}
\caption{Three examples of the cactus construction in Theorem~\ref{thm:cactus}.
The graphs~$G'$ on the right hand side are strongly non-stackable, i.e., not $t$-stackable for any vertex~$t$.}
\label{fig:cactus}
\end{figure}

\begin{proof}
Let $V$ be the vertex set of~$G$.
We can take all degree~1 vertices of~$G'$ as the set~$U$, i.e., $|U|=cn$ and $|W|=|V\setminus U|=n$.
For any vertex~$t$ of~$G'$ we have $|U'|\geq c(n-1)$.
Furthermore, the diameter of~$G'$ equals $d':=d+2$.
Consequently, the inequality $|U'|>(d'-1)|W|$ is satisfied, as $c(n-1)>(d+1)n$ is equivalent to the assumption~$c>(d+1)n/(n-1)$.
\end{proof}

Generalizing this construction, given any graph~$G$, we can glue copies of a large enough complete bipartite graph~$K_{c,c}$ to each vertex of~$G$, which produces a strongly non-stackable graph with arbitrarily large minimum degree.
This yields the following result.

\begin{theorem}
\label{thm:nonstack-mindeg}
For any integer~$c\geq 1$, there is a strongly non-stackable graph~$G$ with minimum degree~$c$.
\end{theorem}

This can be strengthened as follows.

\begin{theorem}
\label{thm:nonstack-conn}
For any integer~$c\geq 1$, there is a strongly non-stackable graph~$G$ that is $c$-connected.
\end{theorem}

\begin{proof}
Such a graph~$G$ can be constructed as follows:
Take a copy of~$K_{c,c}$, and to each of its partition classes, glue one copy of~$K_{c,5c}$ (with its smaller partition class).
We can take all degree~$c$ vertices of~$G'$ as the set~$U$, i.e., $|U|=10c$ and $|W|=2c$.
For any vertex~$t$ of~$G$ we have $|U'|\geq 5c$.
Furthermore, the diameter of~$G'$ equals~$d'=3$.
Consequently, the inequality $|U'|>(d'-1)|W|$ is satisfied, as $5c>2\cdot 2c$ is equivalent to~$5>4$.
\end{proof}

\begin{figure}[h]
\includegraphics[page=2]{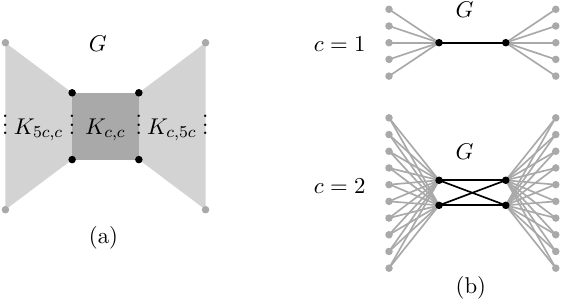}
\caption{Construction of $c$-connected strongly non-stackable graphs explained in the proof of Theorem~\ref{thm:nonstack-conn}.
Part~(a) shows the general construction, and (b) two concrete examples for $c=1$ and~$c=2$.}
\label{fig:conn}
\end{figure}

In view of Lemma~\ref{lem:nonstack}, one may wonder whether there are strongly non-stackable graphs with small independent sets.
Theorem~\ref{thm:spiky} below shows that the answer is yes, and the construction is based on the following lemma.

\begin{lemma}
\label{lem:nonstack-123}
Let $G$ be a graph with diameter at most~3 that has two degree-1 vertices~$u,v$ with a common neighbor, and let $t$ be a target vertex in distance~2 from~$u$ and~$v$.
Then $G$ is not $t$-stackable.
\end{lemma}

\begin{proof}
Let $w$ be the common neighbor of~$u$ and~$v$.
The cups on $u$ and~$v$ cannot move to~$t$ directly, as they have distance~2 from~$t$.
Consequently, each of them moves to~$t$ in a stack of size~2 or~3 (as the diameter is at most~3).
A stack of size~2 formed with a cup from~$u$ or~$v$ has to use the cup on~$w$.
Similarly, a stack of size~3 formed with both cups from~$u$ and~$v$ has to use the cup on~$w$.
However, a stack of size~3 located on~$u,v$ or~$w$ cannot move to~$t$, as $u,v,w$ have distance 2 or~1 from~$t$.
It can also not move anywhere else, as this would make it a stack of size at least~4, and then it could not move anymore. 
Consequently, not both cups on~$u$ and~$v$ move to~$t$ in the same stack, and at least one of them, w.l.o.g., $u$ say, moves in a stack that does not contain the cup from~$w$.
However, this stack must again have size~3 and must be located on~$u$ (as the single cup from~$u$ cannot move across~$w$), which again means that this stack can never move to~$t$.
\end{proof}

The family of graphs referred to in the next theorem is illustrated in Figure~\ref{fig:spiky}.

\begin{theorem}
\label{thm:spiky}
Let $G$ be a graph obtained from a clique of size at least~2 by attaching groups of at least three pendant edges to at least two of the clique vertices.
Then $G$ is strongly non-stackable.
\end{theorem}

\begin{proof}
$G$ has diameter~3, and regardless of the choice of target vertex~$t$ in~$G$, there are always two degree~1 vertices with a common neighbor in distance~2 from~$t$, so Lemma~\ref{lem:nonstack-123} applies.
\end{proof}

\begin{figure}[h]
\includegraphics[page=2]{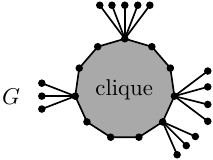}
\caption{Illustration of the graphs in Theorem~\ref{thm:spiky}.}
\label{fig:spiky}
\end{figure}

\section{Stackability is not monotone}
\label{sec:nonmono}

In this section, we show that the property of being stackable is not monotone under adding edges to the graph.

Fay, Hurlbert and Tennant~\cite{MR4749373} completely characterized stackability of complete partite graphs.
We need the following special case of their more general result.
In this statement, $K_{a,b}$ is the complete bipartite graph with partition classes of size~$a$ and~$b$.

\torsten{Check correct theorem numbers from other papers.}

\begin{theorem}[{\cite[Cor.~12]{MR4749373}}]
\label{thm:Kab}
The graph~$K_{a,b}$, $a\leq b$, is stackable if and only if $b=a$ or $b=a+1$.
The graph~$K_{a,b}$, $a+2\leq b$ is not $t$-stackable for any vertex $t$ in the larger partition class, but it is $t$-stackable for any vertex $t$ in the smaller partition class.
\end{theorem}

We note that the non-stackability statement in the previous theorem can be proved by applying Lemma~\ref{lem:nonstack} with~$U$ as the larger partition class.

Veselovac~\cite{veselovac:2022} proved the following result.
Let $F_n$ be the graph formed from a path on $n-2$ vertices by appending two pendant edges to the same end vertex of the path; see Figure~\ref{fig:F10}.

\begin{theorem}[{\cite[Thm~3.10]{veselovac:2022}}]
\label{thm:Fa}
The graph~$F_n$ is stackable for all $n\geq 9$.
\end{theorem}

\begin{theorem}
\label{thm:nonmono}
There are infinitely many graphs~$G\seq H$ on the same vertex set such that $G$ is stackable but $H$ is non-stackable.
\end{theorem}

\begin{proof}
We take $G:=F_n$ for some even integer~$n\geq 10$, which is stackable by Theorem~\ref{thm:Fa}.
The graph~$G$ is bipartite with partition classes of size $n/2-1$ and~$n/2+1$.
Consequently, it is a subgraph of~$K_{n/2-1,n/2+1}$, which by Theorem~\ref{thm:Kab} is non-stackable.
\end{proof}

The previous proof argues that~$F_{10}$ is stackable, whereas~$K_{4,6}\supset F_{10}$ is not.
In fact, Figure~\ref{fig:F10} shows a sequence of five edges~$e_1,\ldots,e_5$ of~$K_{4,6}\setminus F_{10}$ that can be added to~$F_{10}$ one after the other such that the resulting graphs are alternatingly stackable and non-stackable, respectively, which has been checked by computer.

\begin{figure}[h]
\includegraphics[page=1]{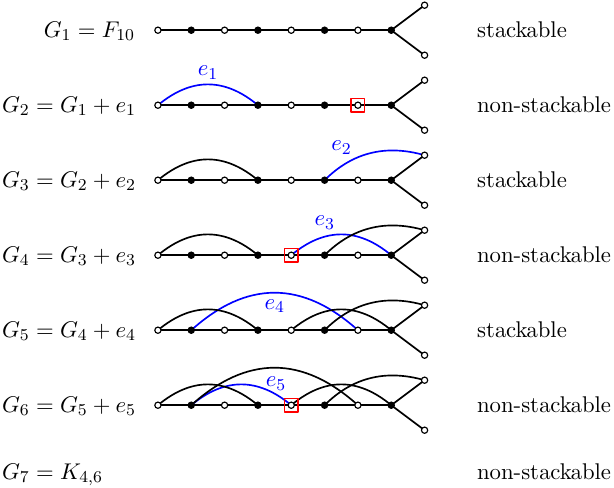}
\caption{Illustration of the non-monotonicity of stackability under adding edges.
For the non-stackable graphs, the vertices $t$ for which the graph is not $t$-stackable are highlighted.}
\label{fig:F10}
\end{figure}

\begin{theorem}
\label{thm:strong-nonmono}
There are two graphs~$G\seq H$ on the same vertex set such that $G$ is stackable but $H$ is strongly non-stackable. 
\end{theorem}

\begin{proof}
Such a pair of graphs~$G$ and~$H$ on 17 vertices is shown in Figure~\ref{fig:strong}.
Specifically, $H$ is obtained from a clique on 11 vertices by appending two triples of pendant edges to two of its vertices, and $G$ is obtained by removing edges from the clique so that it becomes a path between the appendix vertices.
The graph~$G$ is stackable as shown in Figure~\ref{fig:G17}, whereas $H$ is strongly non-stackable by Theorem~\ref{thm:spiky}.
\end{proof}

\begin{figure}[h]
\includegraphics[page=3]{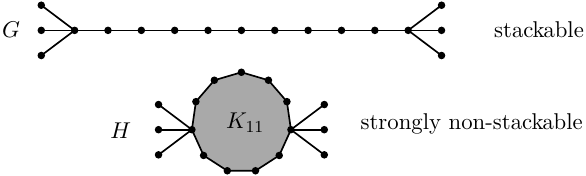}
\caption{Illustration of the graphs in Theorem~\ref{thm:strong-nonmono}.}
\label{fig:strong}
\end{figure}

\begin{figure}[h]
\includegraphics[page=4]{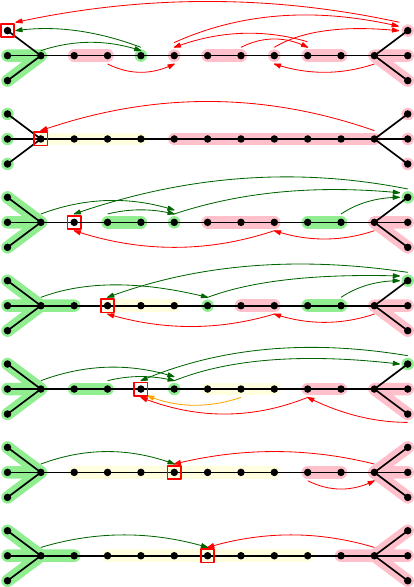}
\caption{Stackability of the graph~$G$ from Theorem~\ref{thm:strong-nonmono}. 
By symmetry, only the 7 target vertices marked by the little squares need to be checked.
Solutions on subpaths or -stars (highlighted by colors) are not shown in full detail for clarity.}
\label{fig:G17}
\end{figure}

\section{Open questions}
\label{sec:open}

We conclude with some challenging open questions.

Can we find other interesting applications for our chunking lemma (Lemma~\ref{lem:chunk})?

Can stackable trees be characterized?

Theorems~\ref{thm:nonmono} and~\ref{thm:strong-nonmono} show that stackability is not monotone under adding edges.
How `bad' can this become?
Specifically, can we find, for each integer $r\geq 3$, a sequence of graphs $G_1\seq G_2\seq \cdots \seq G_r$ on the same vertex set, such that $G_1,G_3,G_5,\ldots$ are stackable and $G_2,G_4,G_6,\ldots$ are (strongly) non-stackable?
Furthermore, are there graphs $G\seq H$ on the same vertex set that differ only in few edges (ideally only one), such that $G$ is stackable and~$H$ is strongly non-stackable?

We may also consider a refined version of the game, where the \defi{weight} of a move is the number of cups that are being moved.
For a given graph~$G$ and target vertex~$t$, we are looking for the sequence of moves with target vertex~$t$ that minimizes the sum of weights of all moves, and we write~$\mu(G,t)$ for this quantity.
We can think of~$\mu(G,t)$ as the total weight of the cups being `transported' during a solution.
Table~\ref{tab:weight} shows this quantity for all target vertices of paths~$P_n$ on vertices $1,\ldots,n$ for all $n\leq 12$.
What can we say about~$\mu(G,t)$ for interesting stackable graphs~$G$?

\begin{table}
\caption{Minimum weight solutions on paths with $n=1,\ldots,12$ vertices.}
\label{tab:weight}
\begin{tabular}{c|c}
$n$ & $\mu(P_n,i)$, $i=1,\ldots,n$ \\\hline
1 & $\{0\}$ \\
2 & $\{1, 1\}$ \\
3 & $\{3, 2, 3\}$ \\
4 & $\{4, 4, 4, 4\}$ \\
5 & $\{6, 5, 6, 5, 6\}$ \\
6 & $\{9, 7, 7, 7, 7, 9\}$ \\
7 & $\{11, 10, 9, 8, 9, 10, 11\}$ \\
8 & $\{12, 12, 12, 10, 10, 12, 12, 12\}$ \\
9 & $\{14, 13, 14, 13, 12, 13, 14, 13, 14\}$ \\
10 & $\{17, 15, 15, 15, 15, 15, 15, 15, 15, 17\}$ \\
11 & $\{19, 18, 17, 16, 17, 18, 17, 16, 17, 18, 19\}$ \\
12 & $\{22, 20, 20, 18, 18, 20, 20, 18, 18, 20, 20, 22\}$ \\
\end{tabular}
\end{table}

From a computational perspective, what is the complexity of deciding whether a given graph~$G$ is stackable, or $t$-stackable for some vertex~$t$?

\bibliographystyle{alpha}
\bibliography{refs}

\end{document}